\documentclass[12pt]{amsart}
\usepackage{amsmath,latexsym,amscd,amsbsy,amssymb,amsfonts,amsthm,fleqn,leqno,
euscript, graphicx, texdraw, pb-diagram}
\usepackage[all]{xy}
\usepackage{color}
\usepackage{tikz}
\numberwithin{equation}{section}

\newtheorem{thm}{Theorem}[section]

\newtheorem{lem}[thm]{Lemma}
\newtheorem{cor}[thm]{Corollary}

\theoremstyle{definition}

\theoremstyle{remark}
\newtheorem{claim}[thm]{Claim}

\begin{document}

\title[On homogeneity of $\mathbb N^\tau$]
{On homogeneity of $\mathbb N^\tau$}

\author{A. Karassev}
\address{Department of Computer Science and Mathematics,
Nipissing University, 100 College Drive, P.O. Box 5002, North Bay,
ON, P1B 8L7, Canada} \email{alexandk@nipissingu.ca}

\author{E. Shchepin}
\address{Steklov Mathematical Institute of Russian Academy of Sciences,
8 Gubkina St. Moscow, 119991, Russia}
\email{scepin@yandex.ru}

\author{V. Valov}
\address{Department of Computer Science and Mathematics,
Nipissing University, 100 College Drive, P.O. Box 5002, North Bay,
ON, P1B 8L7, Canada} \email{veskov@nipissingu.ca}

\thanks{The first author was partially supported by NSERC Grant 257231-14}

\thanks{The third author was partially supported by NSERC
Grant 261914-19.}




\begin{abstract}
It is shown that any homeomorphism between two compact subsets of $\mathbb N^\tau$ can be extended to an autohomeomorphism of $\mathbb N^\tau$.
\end{abstract}

\maketitle
\markboth{}{On homogeneity of $\mathbb N^\tau$}




\section{Introduction}
J. Pollard \cite{jp} established the following theorem: \textit{Let $X$ and $Y$ be complete, nowhere locally compact, zero-dimensional separable metric spaces, and let $P$ and $K$ be closed nowhere dense subsets of $X$ and $Y$, respectively. If $f$ is a homeomorphism between $P$ and $K$, then there exists a homeomorphism between $X$ and $Y$ extending $f$}.

Pollard's result is not anymore true for uncountable powers of the irrationals. For example, if $P$ and $K$ are two homeomorphic closed nowhere dense subsets of $\mathbb N^{\aleph_1}$  such that only one of them is $G_\delta$, then no homeomorphism between $P$ and $K$ admits an extension to an autohomeomorphism on $\mathbb N^{\aleph_1}$. But a non-metrizable analogue of Pollard's theorem remains valid for compact subsets of $\mathbb N^\tau$.
The technique developed in recent papers \cite{sv1} and \cite{sv2} allows to prove the following theorem, where $\pi_1:Y\times\mathbb N^\tau\to Y$ is the projection (by a 0-dimensional space we mean a Tychonoff space having a compactification of covering dimension 0).
\begin{thm} \label{analogue} Let $Y$ be a 0-dimensional Tychonoff space and $P$ and $K$ be compact subsets of $Y\times\mathbb N^\tau$ with $\tau\geq\aleph_0$. Then every homeomorphism $f:P\to K$ with $\pi_1\circ f=\pi_1|P$
 can be extended to a homeomorphism $\widetilde f:Y\times\mathbb N^\tau\to Y\times\mathbb N^\tau$ such that $\pi_1\circ\widetilde f=\pi_1$.
\end{thm}
Theorem 1.1 implies a non-metrizable analogue of mentioned above Pollard's result \cite{jp}.
\begin{cor}
Let $P$ and $K$ be compact subsets of $\mathbb N^\tau$ with $\tau\geq\aleph_0$. Then every homeomorphism between $P$ and $K$ can be extended to a
homeomorphism of $\mathbb N^\tau$.
\end{cor}

\section{Proof of Theorem 1.1}
For any space $X$ let $\mathcal H(X)$ denote the set of all autohomeomorphisms of $X$.
If $d$ be a bounded complete metric on $\mathbb N^{\aleph_0}$ we equip
$\mathcal H(\mathbb N^{\aleph_0})$  with the metric
$\widetilde d=\hat{d}(f,g)+\hat{d}(f^{-1},g^{-1})$, where $\hat{d}(f,g)=\sup\{d(f(x),g(x)):x\in\mathbb N^{\aleph_0}\}$. It is well known that
$\widetilde d$ is a complete metric on $\mathcal H(\mathbb N^{\aleph_0})$.

Next lemma is an analogue of \cite[Lemma 3.1]{sv1}.
\begin{lem}
Let $X$ be a 0-dimensional space and $P, K$ be compact subsets of $X\times\mathbb N^{\aleph_0}$. If $f:P\to K$ and $g\in\mathcal H(X)$ are homeomorphisms with $g\circ\pi_X=\pi_{X}\circ f$, then $f$ can be extended to a homeomorphism
$\widetilde f\in\mathcal H(X\times\mathbb N^{\aleph_0})$ such that $g\circ\pi_X=\pi_{X}\circ\widetilde f$.
\end{lem}
\begin{proof}
Obviously, $g(P_X)=K_X$, where $P_X=\pi_X(P)$ and $K_X=\pi_X(K)$.
Denote by $\pi:X\times\mathbb N^{\aleph_0}\to\mathbb N^{\aleph_0}$ the projection. For any $x\in P_X$ let $\Phi(x)$ be the set of all $h\in\mathcal H(\mathbb N^{\aleph_0})$ such that $f(x,c)=(g(x),h(c))$ for every $c\in\pi_X^{-1}(x)\cap P$. Since $f|(\pi_X^{-1}(x)\cap P)$ is a homeomorphism between the compact subsets $\pi((\{x\}\times\mathbb N^{\aleph_0})\cap P)$ and  $\pi((\{g(x)\}\times\mathbb N^{\aleph_0})\cap K)$
of $\mathbb N^{\aleph_0}$, the Pollard's theorem \cite{jp} cited above yields a
a homeomorphism $h_x\in\mathcal{H}(\mathbb N^{\aleph_0})$ extending $f|(\pi_X^{-1}(x)\cap P)$. Hence,
$\Phi(x)\neq\varnothing$ for all $x\in P_X$.
Moreover, the sets $\Phi(x)$ are closed in $\mathcal H(\mathbb N^{\aleph_0})$ equipped with the metric $\widetilde d$. So, we have a set-valued map
$\Phi:P_X\rightsquigarrow\mathcal H(\mathbb N^{\aleph_0})$. One can show that if $\Phi$ admits a continuous selection $\phi:P_X\to\mathcal H(\mathbb N^{\aleph_0})$,
then the map $f_1:P_X\times\mathbb N^{\aleph_0}\to K_X\times\mathbb N^{\aleph_0}$, defined by $f_1(x,c)=(g(x),\phi(x)(c))$, is a homeomorphism between $P_X\times \mathbb N^{\aleph_0}$ and $K_X\times \mathbb N^{\aleph_0}$ extending $f$ (see \cite[Proposition 2.6.11]{en}) with
$\pi_X\circ f_1=g\circ\pi_X$. On the other hand, since $X$ is 0-dimensional and $P_X$ is a compact subset of $X$, the map $\phi$ has a continuous extension $\widetilde\phi:X\to\mathcal H(\mathbb N^{\aleph_0})$. This is true because $\mathcal H(\mathbb N^{\aleph_0})$ is a metric space.
Indeed, then $\phi(P_X)$ is a compact subset of $\mathcal H(\mathbb N^{\aleph_0})$, hence $\phi(P_X)$ is itself a compact metric space. So,
it is an absolute extensor for 0-dimensional spaces in the sense of Chigogidze \cite{ch}, and we can extend $\phi$ to a map
$\widetilde\phi:X\to\mathcal H(\mathbb N^{\aleph_0})$.
 Next, let $\widetilde f:X\times\mathbb N^{\aleph_0}\to X\times\mathbb N^{\aleph_0}$ be the map defined by $\widetilde f(x,c)=(g(x),\widetilde\phi(x)(c))$. Then $\widetilde f$ is a homeomorphism extending $f$.
Therefore, according to Michael's \cite{em} zero-dimensional selection theorem, it suffices to show that $\Phi$ is lower semi-continuous.

To prove that, let $x^*\in P_X$ be a fixed point and $h^*\in\Phi(x^*)\cap W$, where $W$ is open in $\mathcal H(\mathbb N^{\aleph_0})$. We can assume that
$W$ is of the form $\{h\in\mathcal H(\mathbb N^{\aleph_0}): \{x^*\}\times h(U_i)=\{g(x^*)\}\times V_i{~}, i=1,2,..\}$, where $\{U_i\}_{i=1}^\infty$ and $\{V_i\}_{i=1}^\infty$ are clopen disjoint countable covers of
$\mathbb N^{\aleph_0}$.
Because $P(x^*)=(\{x^*\}\times\mathbb N^{\aleph_0})\cap P$ and $K(g(x^*))=(\{g(x^*)\}\times\mathbb N^{\aleph_0})\cap K$ are compact, there is $k$ such that $(\{x^*\}\times U_i)\cap P(x^*)\neq\varnothing$ and $(\{g(x^*)\}\times V_i)\cap K(g(x^*))\neq\varnothing$ if and only if $i\leq k$.

We extend each of the sets $\{x^*\}\times U_i$ and $\{g(x^*)\}\times V_i$, $i\leq k$, to clopen sets $\widetilde U_i\subset P_X\times\mathbb N^{\aleph_0}$ and $\widetilde V_i\subset K_X\times\mathbb N^{\aleph_0}$ such that
\begin{itemize}
\item[(1)] $\widetilde U_i=O(x^*)\times U_i$ and $\widetilde V_i=g(O(x^*))\times V_i$, where $O(x^*)$ is a clopen neighborhood of $x^*$ in $P_X$
such that $P(x)\subset\bigcup_{i=1}^k\widetilde U_i$ for all $x\in O(x^*)$;
\item[(2)] $O(x^*)$ is so small that $f(\widetilde U_i\cap P)\subset\widetilde V_i\cap K$.
\end{itemize}
We are going to show that for every $x\in O(x^*)$ there exists $h_x\in\Phi(x)\cap W$. We fix such $x$ and observe that all sets
$\widetilde U_i(x)=\widetilde U_i\cap(\{x\}\times\mathbb N^{\aleph_0})$ and $\widetilde V_i(x)=\widetilde V_i\cap(\{g(x)\}\times\mathbb N^{\aleph_0})$ are nowhere locally compact and complete. Moreover, $\widetilde U_i(x)\cap P$ and  $\widetilde V_i(x)\cap K$ are compact sets in $\widetilde U_i(x)$
and $\widetilde V_i(x)$, respectively, and $f^x_i=f|(\widetilde U_i(x)\cap P)$ is a homeomorphism between them for every $i\leq k$.
Hence, by Pollard's theorem \cite{jp}, there exist homeomorphisms $\widetilde f^x_i:\widetilde U_i(x)\to\widetilde V_i(x)$ extending $f^x_i$, $i\leq k$. For every $(x,c)\not\in\bigcup_{i=1}^k\widetilde U_i(x)$ there is exactly one $i>k$ with $c\in U_i$, and we define
$\widetilde f^x_i(x,c)=h^*(x^*,c)$.
The homeomorphisms $\widetilde f^x_i$, $i=1,2,..$, provide a homeomorphism $h_x'$ between $\pi_X^{-1}(x)$ and
$\pi_X^{-1}(g(x))$ extending $f|\pi_X^{-1}(x)\cap P$. Then the equality $h_x(c)=h_x'(x,c)$, $c\in\mathbb N^{\aleph_0}$, defines a homeomorphism
from $\mathcal{H}(\mathbb N^{\aleph_0})$ with $h_x\in\Phi(x)\cap W$. Therefore, $\Phi$ is lower semi-continuous.
\end{proof}

Everywhere below we suppose that $P,K$ are compact subsets of $Y\times \mathbb N^A$ and $f:P\to K$ is a homeomorphism such that
$\pi_1|P=\pi_1\circ f$, where $Y$ is a 0-dimensional space. A set $B\subset A$ is called {\em $f$-admissible} if
there exists a homeomorphism $f_B:P_B\to K_B$ such that  $(f_B\circ p_B)|P=p_B\circ f$ and
${_B}\pi_1|P_B={_B}\pi_1 \circ f_B$, where $\pi_B:\mathbb N^A\to\mathbb N^B$, $p_B:Y\times \mathbb N^A\to Y\times \mathbb N^B$,
and ${_B}\pi_1:Y \times \mathbb N^B\to Y$ denote the projections, $P_B=p_B(P)$ and $K_B=p_B(K)$.
\begin{lem}
If $A$ is an uncountable set, then for every $\alpha\in A$ there is an $f$-admissible countable set $B(\alpha)\subset A$ containing $\alpha$
\end{lem}
\begin{proof}
Since $\pi_1|P=\pi_1\circ f$, $f(y,x)=(y, h_1(y,x))$ for all $(y,x)\in P$, where $h_1:P\to\mathbb N^A$.
Similarly, $f^{-1}(y,x)=(y, h_2(y,x))$ for all $(y,x)\in K$ with $h_2:K\to\mathbb N^A$.

\begin{claim}
For every $i=1,2$ and a countable set
$C\subset A$, there is countable $D(i)\subset A$ containing $C$ and a continuous maps $g_1:P_{D(1)}\to\mathbb N^C$, $g_2:K_{D(2)}\to\mathbb N^C$
with
$(g_1\circ p_{D(1)})|P=\pi_C\circ h_1$ and $(g_2\circ p_{D(2)})|K=\pi_C\circ h_2$.
\end{claim}
Let $\mathcal B$ be a countable base for $\mathbb N^C$. Then $G_U=h_1^{-1}(\pi_C^{-1}(U))$ is a $\sigma$-compact open subset of $P$ for every $U\in\mathcal B$. So, there is a sequence $\{W_U(n)\}_{n\geq 1}$ of standard open sets in $Y\times \mathbb N^A$ such that $G_U$ is the union of all $W_U(n)\cap P$, $n\geq 1$. Therefore, for every $U$ there exists a countable set $C_U\subset A$ with
$p_{C_U}^{-1}(p_{C_U}(W_U(n)))=W_U(n)$, $n\geq 1$. We can assume that each $C_U$ contains $C$. Then $D(1)=\bigcup_{U\in\mathcal B}C_U$ is a countable set containing $C$ and $p_{D(1)}(y,x)=p_{D(1)}(y',x')$ implies $\pi_C(h_1(y,x))=\pi_C(h_1(y',x'))$, where $(y,x), (y',x')\in P$.
Because $P_{D(1)}$ is compact, this yields the existence of a map $g_1:P_{D(1)}\to\mathbb N^C$ with $(g_1\circ p_{D(1)})|P=\pi_C\circ h_1$.
Similarly, we can find a countable set $D(2)\subset A$ which contains $C$ and a map $g_2:K_{D(2)}\to\mathbb N^C$ satisfying the claim.

Using Claim 2.3, we construct by induction an increasing sequence $\{B(n)\}_{n\geq 0}$ of countable sets $B(n)\subset A$ and maps
$\varphi_n:P_{B(n+1)}\to \mathbb N^{B(n)}$ for $n=2k$ and $\psi_n:K_{B(n+1)}\to \mathbb N^{B(n)}$ for $n=2k+1$ such that
\begin{itemize}
\item $B(0)=\{\alpha\}$;
\item $\pi_{B(n)}\circ h_1=(\varphi_n\circ p_{B(n+1)})|P$ if $n=2k$;
\item $\pi_{B(n)}\circ h_2=(\psi_n\circ p_{B(n+1)})|K$ if $n=2k+1$.
\end{itemize}
Let $B(\alpha)=\bigcup_{n=0}^\infty B(n)$. Then we have maps $\varphi_{B(\alpha)}:P_{B(\alpha)}\to \mathbb N^{B(\alpha)}$ and
$\psi_{B(\alpha)}:K_{B(\alpha)}\to \mathbb N^{B(\alpha)}$
such that $\pi_{B(\alpha)}\circ h_1=(\varphi_{B(\alpha)}\circ p_{B(\alpha)})|P$ and
$\pi_{B(\alpha)}\circ h_2=(\psi_{B(\alpha)}\circ p_{B(\alpha)})|K$.
Observe that $P_{B(\alpha)}$ and $K_{B(\alpha)}$ are subsets of $Y\times\mathbb N^{B(\alpha)}$. Then $f_{B(\alpha)}:P_{B(\alpha)}\to Y\times\mathbb N^{B(\alpha)}$, defined by $f_{B(\alpha)}(y,x)=(y,\varphi_{B(\alpha)}(y,x))$ for every $(y,x)\in P_{B(\alpha)}$, is a homeomorphism between
$P_{B(\alpha)}$ and $K_{B(\alpha)}$ whose inverse is the map $g_{B(\alpha)}:K_{B(\alpha)}\to Y\times\mathbb N^{B(\alpha)}$, defined by $g_{B(\alpha)}(y,x)=(y,\psi_{B(\alpha)}(y,x))$ for $(y,x)\in K_{B(\alpha)}$. Therefore, $B(\alpha)$ is $f$-admissible.
\end{proof}

{\em Proof of Theorem $1.1$.}
We identify $\mathbb N^\tau$ with $\mathbb N^A$, where $A$ is a set of cardinality $\tau$.
The case $\tau=\aleph_0$ follows from Lemma 2.1.
So, let $A=\{\alpha:\alpha<\omega(\tau)\}$ be uncountable.
By Lemma 2.2, we can cover $A$ by a family $\{B(\alpha):\alpha<\omega(\tau)\}$ of countable $f$-admissible sets. Since any union of $f$-admissible sets is also $f$-admissible, from the family $\{B(\alpha):\alpha<\omega(\tau)\}$ we obtain an increasing family of $f$-admissible sets $A(\alpha)$ and homeomorphisms
$f_{A(\alpha)}:P_{A(\alpha)}\to K_{A(\alpha)}$ such that:

\begin{itemize}
\item[(3)]$A(1)$ is countable, and the cardinality of each $A(\alpha)$ is less than $\tau$;
\item[(4)]$A(\alpha)=\bigcup_{\beta<\alpha}A(\beta)$ if $\alpha$ is a limit ordinal;
\item[(5)]$A(\alpha+1)\backslash A(\alpha)$ is countable but infinite for all $\alpha$;
\item[(6)]$\displaystyle p_{A(\alpha)}\circ f=(f_{A(\alpha)}\circ p_{A(\alpha)})|P$.
\end{itemize}
We need to prove that each $f_{A(\alpha)}$ can be extended to a homeomorphism
$\widetilde f_{A(\alpha)}:Y\times \mathbb N^{A(\alpha)}\to Y\times\mathbb N^{A(\alpha)}$ such that
\begin{itemize}
\item[(7)] $p_{A(\alpha)}^{A(\alpha+1)}\circ\widetilde f_{A(\alpha+1)}=(\widetilde
f_{A(\alpha)}\circ p_{A(\alpha)}^{A(\alpha+1)})$;
\item[(8)] ${_\alpha}\pi_1={_\alpha}\pi_1\circ\widetilde f_{A(\alpha)}$, where ${_\alpha}\pi_1:Y\times \mathbb N^{A(\alpha)}\to Y$ denotes the projection.
\end{itemize}

The proof is by transfinite induction.
The first extension $\widetilde f_{A(1)}$ exists by Lemma 2.1 because $P_{A(1)}$ and $K_{A(1)}$ are compact subsets of $Y\times \mathbb N^{A(1)}$ and ${_1}\pi_1|P_{A(1)}={_1}\pi_1\circ f_{A(1)}$.
If $\beta$ is a limit ordinal and $\widetilde f_{A(\alpha)}$ is already defined for all $\alpha<\beta$, then item (4) implies the existence of
$\widetilde f_{A(\beta)}$. Therefore, we need only to define $\widetilde f_{A(\alpha+1)}$ provided $\widetilde f_{A(\alpha)}$ exists.

To this end, we apply again Lemma 2.1 for the space
$Y\times \mathbb N^{A(\alpha+1)}=Y\times \mathbb N^{A(\alpha)}\times\mathbb N^{A(\alpha+1)\backslash A(\alpha)}$, the sets
$P_{A(\alpha+1)}, K_{A(\alpha+1)}$, the projection
$\pi:Y\times \mathbb N^{A(\alpha)}\times\mathbb N^{A(\alpha+1)\backslash A(\alpha)}\to Y\times \mathbb N^{A(\alpha)}$, and the homeomorphisms $f_{A(\alpha+1)}$ and  $\widetilde f_{A(\alpha)}$. Moreover, $Y\times \mathbb N^{A(\alpha)}$ has a 0-dimensional compactification because both $Y$ and $N^{A(\alpha)}$ have such compactifications.
Hence, there is a homeomorphism
$\widetilde f_{A(\alpha+1)}\in\mathcal H(Y\times \mathbb N^{A(\alpha+1)})$ extending $f_{A(\alpha+1)}$ and satisfying condition $(7)$.
Since ${_\alpha}\pi_1 ={_\alpha}\pi_1\circ\widetilde f_{A(\alpha)}$, we also have ${_{\alpha+1}}\pi_1 ={_{\alpha+1}}\pi_1\circ\widetilde f_{A(\alpha+1)}$. $\Box$

{\em Proof of Corollary $1.2$.} The corollary is obtained from Theorem 1.1 by letting $Y$ to be the one-point space.   $\Box$



\end{document}